\documentclass[11pt,reqno]{amsart}
\usepackage{amsmath,amssymb,amsthm,amsfonts,enumerate,hyperref}

\theoremstyle{plain}
\newtheorem{theorem}{Theorem}[section]
\newtheorem{proposition}[theorem]{Proposition}
\newtheorem{lemma}[theorem]{Lemma}
\newtheorem*{assumption}{Assumption}
\newtheorem*{thmA}{Theorem A}
\newtheorem*{thmB}{Theorem B}
\newtheorem*{thmC}{Theorem C}
\newtheorem*{corD}{Corollary D}
\newtheorem*{thmE}{Theorem E}
\newtheorem*{thmF}{Theorem F}

\theoremstyle{definition}
\newtheorem{definition}[theorem]{Definition}

\theoremstyle{remark}
\newtheorem{remark}[theorem]{Remark}
\newtheorem*{technical}{Technical aspects of the proofs}

\numberwithin{equation}{section}

\newcommand{\eps}{\varepsilon}
\newcommand{\Lap}{\Delta}
\newcommand{\Rm}{\mathrm{Rm}}
\newcommand{\Rc}{\mathrm{Rc}}
\newcommand{\Hess}{\mathrm{Hess}}
\newcommand{\tr}{\mathrm{tr}}
\newcommand{\grf}{g_{\mathrm{RF}}}
\newcommand{\norm}[1]{\lVert#1\rVert}
\renewcommand{\div}{\mathrm{div}}

\DeclareMathOperator{\aint}{\int\!\!\!\!\!\! \rule[2.6pt]{8pt}{.4pt}}

\begin{document}

\title{Perelman's lambda-functional and the stability of Ricci-flat metrics}

\author{Robert Haslhofer}

\address{Department of Mathematics, ETH Z\"{u}rich, Switzerland}

\email{robert.haslhofer@math.ethz.ch}

\begin{abstract}
In this article, we introduce a new method (based on Perelman's $\lambda$-functional) to study the stability of compact Ricci-flat metrics. Under the assumption that all infinitesimal Ricci-flat deformations are integrable we prove: (A) a Ricci-flat metric is a local maximizer of $\lambda$ in a $C^{2,\alpha}$-sense if and only if its Lichnerowicz Laplacian is nonpositive, (B) $\lambda$ satisfies a {\L}ojasiewicz-Simon gradient inequality, (C) the Ricci flow does not move excessively in gauge directions. As consequences, we obtain a rigidity result, a new proof of Sesum's dynamical stability theorem, and a dynamical instability theorem.
\end{abstract}

\maketitle

\section{Introduction}
A Ricci-flat manifold is a Riemannian manifold with vanishing Ricci curvature. Compact Ricci-flat manifolds are fairly hard to find, and their properties are of great interest (see \cite{Be,Jo} for extensive information). They are the critical points of the Einstein-Hilbert functional and the fixed points of Hamilton's Ricci flow \cite{Ha},
\begin{equation}\label{ricciflow}
\partial_tg(t)=-2\Rc_{g(t)}, \qquad g(0)=g_0.
\end{equation}
Historically, since Ricci-flat metrics are saddle points (but not extrema) of the Einstein-Hilbert functional and since the Ricci flow is not a gradient flow in the strict sense, the variational interpretation of Ricci-flatness was rather obscure. However, Perelman made the remarkable discovery that the Ricci flow can be interpreted as gradient flow of the functional
\begin{equation}\label{lambda}
\lambda(g)=\inf_{\begin{subarray}{c}
f\in C^\infty(M)\\
\int_M e^{-f}dV_g=1
\end{subarray}}
\int_M\left(R_g+|D f|_g^2\right)e^{-f}dV_g
\end{equation}
 on the space of metrics modulo diffeomorphisms.
 In particular, $\lambda$ is nondecreasing under the Ricci flow and the stationary points of $\lambda$ are precisely the Ricci-flat metrics \cite{Pe,KL}. The second variation of $\lambda$ is given in terms of the Lichnerowicz Laplacian \cite{CHI}.\\
 
We will be concerned with the \emph{stability} of compact Ricci-flat metrics $\grf$. To discuss this properly, let us consider the following notions of stability:
\begin{enumerate}[i.]
\item\emph{(Dynamical stability)}\label{dynstab} For every neighborhood $\mathcal{V}$ of $\grf$ in the space of metrics there exists a smaller neighborhood $\mathcal{U}\subset\mathcal{V}$ such that the Ricci flow starting in $\mathcal{U}$ exists and stays in $\mathcal{V}$ for all $t\geq 0$ and converges to a Ricci-flat metric in $\mathcal{V}$.\footnote{This notion of stability was called weak dynamical stability in \cite{Se}. However, it is the strongest possible notion of stability for dynamical systems with non-isolated critical points where one can only hope to prove convergence to \emph{some} critical point close to the specified one. Since Ricci-flat metrics are always non-isolated critical points in the space of metrics modulo diffeomorphisms, we decided to simply drop the word weak.}
\item\emph{(Local maximum of $\lambda$)}\label{lamstab} There exists a neighborhood $\mathcal{U}$ of $\grf$ such that $\lambda(g)\leq 0$ for all $g\in\mathcal{U}$ with equality if and only if $g$ is Ricci-flat.
\item\emph{(Linear stability)}\label{linstab} All eigenvalues of the Lichnerowicz Laplacian $\Lap^L_{\grf}=\Lap_{\grf}+2\Rm_{\grf}$ are nonpositive.
\end{enumerate}
It is easy to see that \ref{dynstab}$\Rightarrow$\ref{lamstab}$\Rightarrow$\ref{linstab}. Conversely, Natasa Sesum proved that linear stability implies dynamical stability if all infinitesimal Ricci flat deformations are integrable \cite{Se}. This was not straightforward, but using the integrability condition she succeeded in finding a good sequence of new reference metrics for the Ricci-DeTurck flow. In particular, she proved the dynamical stability of the $\mathrm{K}3$ surface, which already had been conjectured and partly proven by Guenther-Isenberg-Knopf \cite{GIK}. Additional interesting results about the dynamical stability of the Ricci flow can be found in \cite{Ba,Kn,KY,LY,OW,SSS1,SSS2,Ye}, and the K\"{a}hler case has also been studied by various authors. The proofs are mainly based on the Ricci-DeTurck flow, its linearization and parabolic estimates.\\

In this article, we introduce a \emph{new method} inspired by the work of Leon Simon \cite{Si} and the work of Dai-Wang-Wei \cite{DWW}, see also \cite{CFS, CHI, Ra,TZ}. We use the $\lambda$-functional to study stability and instability. With this method we obtain a new proof of Sesum's dynamical stability result (Theorem E) and a number of new results: In particular, we prove a {\L}ojasiewicz-Simon gradient inequality for the $\lambda$-functional (Theorem B), and a transversality estimate (Theorem C). Moreover, we prove a local analogue of the positive mass theorem for some compact Ricci-flat metrics (Theorem A), and the corresponding rigidity result (Corollary D). We also prove that unstable Ricci-flat metrics give rise to nontrivial ancient Ricci flows emerging from them (Theorem F). In addition to these results, which we hope are of independent interest, the focus is on methods and proofs. We believe that it is important to understand stability in terms of the $\lambda$-functional and not just in terms of PDEs and that our proofs shed new light on the variational structure of Ricci-flat metrics and the role of the gauge group.\\

The logical structure is that we have three general theorems (A,B,C) and three consequences (D,E,F). To state them, let us fix the following assumption for the whole paper:

\begin{assumption}
Let $(M,\grf)$ be a compact, Ricci-flat manifold and assume that all infinitesimal Ricci-flat deformations of $\grf$ are integrable.
\end{assumption}

The integrability condition means that for every symmetric 2-tensor $h$ in the kernel of the linearization of Ricci, we can find a curve of Ricci-flat metrics with initial velocity $h$ (see \cite[Sec. 12]{Be} and Section \ref{ebinandintegrability} for details, and the discussion after Theorem C for applicability and context).\\

To set the stage for Theorem A, recall that at a Ricci-flat metric $\lambda(\grf)=0$, $D\lambda({\grf})=0$ and \cite{CHI}
\begin{equation}
D^2\lambda({\grf})[h,h]=\tfrac{1}{2}\aint_M\langle h, \Lap_{\grf}^L h\rangle_{\grf}dV_{\grf},\quad h\in\ker\div_{\grf},
\end{equation}
where $\Lap^L_{\grf}h_{ij}=\Lap h_{ij}+2R_{ipjq}h_{pq}$. Thus, as mentioned above, local maxima of $\lambda$ are linearly stable. In the \emph{integrable} case, we can prove the converse implication:

\begin{thmA}[Local maxima of $\lambda$]\label{locmax} If the Lichnerowicz Laplacian $\Lap^L_{\grf}=\Lap_{\grf}+2\Rm_{\grf}$ is nonpositive, then there exists a $C^{2,\alpha}$-neighborhood\, $\mathcal{U}\subset\mathcal{M}(M)$ of $\grf$ in the space of metrics on $M$, such that $\lambda(g)\leq 0$ for all $g\in\mathcal{U}$. Moreover, equality holds if and only if $g$ is Ricci-flat.
\end{thmA}

Theorem A is nontrivial for the following three reasons: $D^2\lambda$ vanishes on Lie-derivatives, $\Lap^L$ always has a kernel, and  it is difficult to estimate the error term in the Taylor-expansion coming from the third variation of $\lambda$.\\

Next, let us state our {\L}ojasiewicz-Simon gradient inequality for the $\lambda$-functional (with optimal {\L}ojasiewicz exponent $1/2$ due to integrability):

\begin{thmB}[{\L}ojasiewicz inequality for $\lambda$]\label{lojasiewicz} There exists a $C^{2,\alpha}$-neighborhood\, $\mathcal{U}\subset\mathcal{M}(M)$ of $\grf$ and a constant $c=c(M,\grf)>0$ such that
\begin{equation}
\norm{\Rc_g+\Hess_gf_g}_{L^2}\geq c |\lambda(g)|^{1/2}\label{lojaineq}
\end{equation}
for all $g\in\mathcal{U}$, where $f_g$ is the minimizer in (\ref{lambda}).
\end{thmB}

For the interpretation of (\ref{lojaineq}) as a gradient inequality note that $\Rc_g+\Hess_gf_g$ is the (negative) $L^2(M,e^{-f_g}dV_g)$-gradient of $\lambda$ by Perelman's first variation formula
\begin{equation}
D\lambda(g)[h]=- \int_M\langle h,\Rc_g+\Hess_g f_g \rangle_g e^{-f_g}dV_g.\label{perfirstvar}
\end{equation}
Theorem B is interesting, since it can be used as a general tool to prove convergence and to draw further dynamical conclusions. More precisely, we will always apply it in combination with the following theorem:

\begin{thmC}[Transversality]\label{gaugeest} There exists a $C^{2,\alpha}$-neighborhood\, $\mathcal{U}\subset\mathcal{M}(M)$ of $\grf$ and a constant $c=c(M,\grf)>0$ such that
\begin{equation}
\norm{\Rc_g+\Hess_gf_g}_{L^2}\geq  c \norm{\Rc_g}_{L^2}\label{transverest}
\end{equation}
for all $g\in\mathcal{U}$, where $f_g$ is the minimizer in (\ref{lambda}).
\end{thmC}

Theorem C is a quantitative generalization of the fact that compact steady solitons are Ricci-flat. Since $\div_f(\Rc+\Hess f)=0$ \cite[Eq. 10.11]{KL}, it shows that the Ricci flow does not move excessively in gauge directions.\\

Before turning to the applications, let us discuss what is currently known and unknown: All \emph{known} compact Ricci-flat manifolds satisfy the integrability assumption and $\Lap^L\leq 0$ (they have special holonomy, so this follows from the results in \cite{DWW,Jo,Ti,To,Wa,Ya}), but it is a major open question what the true landscape of all compact Ricci-flat manifolds looks like. One main open problem is to construct a compact Ricci-flat manifold with holonomy full $\mathrm{SO}_n$ (see e.g. \cite[Sec. 0.I]{Be}). Another related open question (called the positive mass problem for Ricci flat manifolds in \cite{CHI}) is if unstable \emph{compact} Ricci-flat metrics exist. Finally, one can ask if there is a Ricci-flat metric with nonintegrable deformations.\\

Given the difficulty of the above questions and that our picture of compact Ricci-flat metrics already has been drastically changed twice due to Yau and Joyce, we find it very interesting to discuss all cases. First, as an immediate consequence of Theorem A we obtain (compare with \cite{DWW}):

\begin{corD}[Rigidity of Ricci-flat metrics]\label{rigidity} If $\Lap^L_{\grf}\leq 0$, then every small deformation of $\grf$ with nonnegative scalar curvature is Ricci-flat.
\end{corD}

The reader might wish to compare Corollary D with the following rigidity case of the positive mass theorem \cite{SY,Wi}: Every compact deformation of the flat metric on $\mathbb{R}^n$ with nonnegative scalar curvature is flat. More generally, Theorem A can be thought of as the positive mass theorem for linearly stable, integrable, compact Ricci-flat metrics.\\

Second, as mentioned before, we have a new proof of Sesum's dynamical stability theorem:

\begin{thmE}[Dynamical stability]\label{dynstabthm} If $k\geq 3$ and $\Lap^L_{\grf}\leq 0$, then for every $C^k$-neighborhood $\mathcal{V}$ of ${\grf}$ there exists a $C^{k+2}$-neighborhood $\mathcal{U}\subset\mathcal{V}$ of ${\grf}$ such that the Ricci flow starting in  $\mathcal{U}$ exists and stays in $\mathcal{V}$ for all $t\geq 0$ and converges exponentially to a Ricci-flat metric in $\mathcal{V}$.
\end{thmE}

Our proof is based on Theorem A, B and C, and shows that the energy controls the distance (Lemma \ref{energylemma}).\\

Third, using Theorem B and C we obtain:

\begin{thmF}[Dynamical instability]\label{ancientsolem} If $\Lap^L_{\grf}\nleq 0$, then there exists a nontrivial ancient solution emerging from $\grf$, i.e. a nontrivial Ricci flow $g(t), t\in(-\infty,T)$ with $\lim_{t\to -\infty} g(t)=\grf$.
\end{thmF}

Note that the statement of Theorem F is much sharper than just some sort of instability/nonconvergence of flows starting nearby. Together with Theorem E it gives a quite complete picture of what could happen in the compact integrable case.\\

It should be straightforward to generalize Theorem F and our proof based on the {\L}ojasiewicz inequality to other flows. A generalization to the noncompact case would be very interesting, since for example the Riemannian Schwarzschild metric is linearly unstable \cite[Sec. 5]{GPY}.\\

There are various further applications of Theorem B and C. For example, the reader might wish to prove a dichotomy theorem in the case where $\lambda$ is not a local maximum, i.e. the flow starting near such a Ricci-flat metric either converges or runs away (compare with \cite[Thm. 2]{Si}). Finally, the nonintegrable case is discussed in Remark \ref{technicalremark}.

\begin{remark}
It suffices to check the condition $\Lap^L\leq 0$ on TT, i.e. on transverse traceless symmetric 2-tensors, since $\Lap^L$ is always nonpositive on the other components \cite{GIK}.
\end{remark}

\begin{technical}
We take care of the gauge directions using the Ebin-Palais slice theorem and of the kernel of $\Lap^L$ using the integrability assumption. The main technical step in the proof of Theorem A, is the estimate
\begin{equation}
\left|\tfrac{d^3}{d\eps^3}|_0\lambda(g+\eps h)\right|\leq C \norm{h}_{C^{2,\alpha}}\norm{h}_{H^1}^2\label{errorcubic}
\end{equation}
uniformly in a $C^{2,\alpha}$-neighborhood of $\grf$ (Proposition \ref{3rdvaroflambda}). This allows us to conclude that $\lambda$ is indeed maximal, since
\begin{equation}
D^2\lambda({\grf})[h,h]\leq -c  \norm{h}_{H^1}^2\label{partinttrick}
\end{equation}
on the space normal to the flat directions.\\
Regarding Theorem B, C and E, let us just emphasize that it was not at all straightforward to adapt Leon Simon's methods to the Ricci flow. For the numerous technical problems and their solutions we refer the reader to Section \ref{maintheorems} and Section \ref{stabandinstab}.  In particular, with Theorem C we find a way to handle the $\Hess_gf_g$-term, a term that is the source of many difficulties. The technical heart consists of Lemmas \ref{taylorsteady} and \ref{estimateofo}.\\
Finally, the Ricci flow in Theorem F is constructed by a suitable limiting process, and the main step is to prove that this limit is \emph{nontrivial}.\\
\end{technical}

This article is organized as follows: In Section \ref{variationalstructure}, we analyze the variational structure of $\lambda$, in particular, we prove (\ref{errorcubic}). In Section \ref{ebinandintegrability}, we recall some facts about the Ebin-Palais slice theorem and integrability. In Section \ref{maintheorems}, we prove A, B, C and D. Finally, as a consequence, we obtain the stability and instabilty results E and F in Section \ref{stabandinstab}.\\

\emph{Acknowledgements.} I would like to thank Tom Ilmanen for many interesting discussions, 
in particular for suggesting the {\L}ojasiewicz-Simon type argument. Moreover, I would like to thank Michael Struwe for his support, Richard Bamler and Reto M\"{u}ller for detailed comments on a preliminary version of this paper, and the Swiss National Science Foundation for partial financial support.

\section{The variational structure}\label{variationalstructure}
We will analyze the variational structure of $\lambda$ using eigenvalue perturbation theory \cite[Sec. XII]{RS}.\\
Let $(M,g)$ be a compact Riemannian manifold. Substituting $w=e^{-f/2}$ in (\ref{lambda}), we see that $\lambda(g)$ is the smallest eigenvalue of the Schr\"{o}dinger operator $H_{g}=-4\Lap_{g}+R_{g}$. The spectrum of $H_{g}$ consists only of real eigenvalues of finite multiplicity $\lambda(g)=\lambda_1(g)<\lambda_2(g)\leq\lambda_3(g)\leq\ldots$ tending to infinity and the smallest eigenvalue is simple. From the minimax characterization
\begin{equation}
\lambda_k(g)=\min_{\begin{subarray}{c}
W\subset C^\infty(M)\\
\dim W=k
\end{subarray}}\max_{\begin{subarray}{c}
w\in W\\
w\neq 0
\end{subarray}}\frac{\int_M\left(4|D w|_g^2+R_gw^2\right)dV_g}{\int_Mw^2dV_g},
\end{equation}
we see that $\lambda_k\!:\!\mathcal{M}(M)\to\mathbb{R}$ is continuous with respect to the $C^2$-topology on the space of metrics on $M$.
Along a variation, $g(\eps)=g+\eps h$, the smallest eigenvalue $\lambda(g(\eps))$ depends analytically on $\eps$ \cite[Sec. 7.I.2.2]{KL}. To analyze this $\eps$-dependence, it is convenient to study the resolvent $(\lambda-H_{g(\eps)})^{-1}$, defined for complex $\lambda$ outside the spectrum. Observe that\begin{equation}\label{pertform1}
P_{g(\eps)}=\tfrac{1}{2\pi i}\oint_{|\lambda-\lambda(g)|=r}(\lambda-H_{g(\eps)})^{-1}d\lambda
\end{equation}
is the projection to the one-dimensional $\lambda(g(\eps))$-eigenspace of $H_{g(\eps)}$. Here, $r$ is assumed to be large enough to encircle $\lambda(g(\eps))$, but small enough to stay away from the other eigenvalues. Thus
\begin{equation}
H_{g(\eps)} P_{g(\eps)} w_g=\lambda(g(\eps)) P_{g(\eps)} w_g,
\end{equation}
where $w_g$, called the ground-state in the following, is the unique positive $L^2(M,dV_g)$-normalized eigenfunction of $H_g$ with eigenvalue $\lambda(g)$. Thus, for small $\eps$, we obtain
\begin{equation}\label{pertform2}
\lambda(g(\eps))=\lambda(g)+\frac{\langle w_g ,(H_{g(\eps)}-H_g)P_{g(\eps)} w_g\rangle_{L^2(M,dV_g)}}{\langle w_g,P_{g(\eps)} w_g\rangle_{L^2(M,dV_g)}}.
\end{equation}
\begin{lemma}\label{perturbationlemma}
Let $(M,g)$ be a compact Riemannian manifold and $h$ a symmetric 2-tensor. Then the smallest eigenvalue $\lambda(g+\eps h)$ of the operator $H_{g+\eps h}=-4\Lap_{g+\eps h}+R_{g+\eps h}$ depends analytically on $\eps$ and the first three derivatives are given by the following formulas:
\begin{align}
&\tfrac{d}{d\eps}|_0\lambda(g+\eps h)=\langle w,H'[h] w\rangle,\label{pertseries1}\\
&\tfrac{d^2}{d\eps^2}|_0\lambda(g+\eps h)\nonumber\\
&\qquad=\langle w, H''[h,h]w\rangle+\tfrac{2}{2\pi i}\oint\langle w,H'[h](\lambda-H)^{-1}H'[h] w\rangle\tfrac{d\lambda}{\lambda-\lambda(g)},\label{pertseries2}\\
&\tfrac{d^3}{d\eps^3}|_0\lambda(g+\eps h)\nonumber\\
&\qquad\begin{array}{l}
=\langle w,H'''[h,h,h] w\rangle\\
\quad+\tfrac{6}{2\pi i}\oint \langle w,H'[h](\lambda-H)^{-1}H'[h](\lambda-H)^{-1} H'[h] w\rangle \tfrac{d\lambda}{\lambda-\lambda(g)}\\
\quad+\tfrac{3}{2\pi i}\oint \langle w, H'[h](\lambda-H)^{-1}H''[h,h] w\rangle \tfrac{d\lambda}{\lambda-\lambda(g)}\\
\quad+\tfrac{3}{2\pi i}\oint \langle w, H''[h,h](\lambda-H)^{-1}H'[h] w\rangle \tfrac{d\lambda}{\lambda-\lambda(g)}\\
\quad-\langle w,H'[h]w\rangle\tfrac{6}{2\pi i}\oint \langle w,H'[h](\lambda-H)^{-1} H'[h] w\rangle \tfrac{d\lambda}{(\lambda-\lambda(g))^2}.
\end{array}\label{pertseries3}
\end{align}
Here $w=w_g$ is the ground state of $H=H_g=-4\Lap_g+R_g$ and $H^{(k)}[h,\ldots,h]=\tfrac{d^k}{d\eps^k}|_0(-4\Lap_{g+\eps h}+R_{g+\eps h})$. The complex integrals are over a small circle around $\lambda(g)$ and $\langle\; , \, \rangle$ denotes the $L^2(M,dV_g)$ inner product.
\end{lemma}

The proof of Lemma \ref{perturbationlemma} can be found in Appendix \ref{proofpertlemma}, but let us illustrate here, where (\ref{pertseries2}) comes from. We differentiate (\ref{pertform2}) twice. To get a nonzero contribution when evaluated at $\eps=0$ the derivative has to hit $H_{g(\eps)}$ at least once, thus
\begin{multline}
\tfrac{d^2}{d\eps^2}|_0\lambda(g+\eps h)\\
=\langle w, H''[h,h]w\rangle+2\langle w, H'[h]P'[h]w\rangle-2\langle w, H'[h]w\rangle \langle w, P'[h]w\rangle,
\end{multline}
where we also used $P_gw_g=w_g$ and $\langle w_g,w_g\rangle_{L^2(M,dV_g)}=1$. Differentiating (\ref{pertform1}) and taking care of the operator ordering, we obtain
\begin{equation}
P'[h]=\tfrac{1}{2\pi i}\oint_{|\lambda-\lambda(g)|=r}(\lambda-H)^{-1}H'[h](\lambda-H)^{-1}d\lambda.
\end{equation}
Using $(\lambda-H)^{-1}w=(\lambda-\lambda(g))^{-1}w$ and the fact that $H$ is symmetric with respect to the $L^2$ inner product, Equation (\ref{pertseries2}) follows. In particular, observe that $\langle w, P'[h]w\rangle$ vanishes, since $\oint (\lambda-\lambda(g))^{-2}d\lambda=0$. The computation for (\ref{pertseries1}) and (\ref{pertseries3}) is similar and the differentiability and convergence can be justified, see Appendix \ref{proofpertlemma} for details.

From the usual formulas for the variation of the Laplacian and the scalar curvature (see e.g. \cite[Sec. 1.K]{Be}), we obtain
\begin{equation}\label{hprime}
H'[h]=4h\!:\!D^2+4\div h\!:\!D-2D\text{tr}h\!:\!D-\langle h,\Rc\rangle+\div\div h -\Lap\text{tr}h.
\end{equation}
Inserting this in (\ref{pertseries1}), substituting $w=e^{-f/2}$ and using partial integration gives Perelman's first variation formula
\begin{equation}\label{1stvaroflambda}
D\lambda(g)[h]=- \int_M\langle h,\Rc_g+\Hess_g f_g \rangle_g e^{-f_g}dV_g,
\end{equation}
where $f_g$ is the minimizer in (\ref{lambda}). Due to diffeomorphism invariance, $D\lambda(g)$ vanishes on Lie-derivatives, in particular
\begin{equation}\label{dlambdadiffinv}
\int_M\langle \Hess_g f_g,\Rc_g+\Hess_g f_g \rangle_g e^{-f_g}dV_g=0.
\end{equation}
The stationary points of $\lambda$ are precisely the Ricci-flat metrics (compactness is crucial here). At a Ricci-flat metric $\grf$ we have $\lambda(\grf)=0, D\lambda(\grf)=0$ and (compare with \cite{CHI})
\begin{equation}\label{2ndvaroflambda}
D^2\lambda(\grf)[h,h] = \left\{ \begin{array}{ll}
\frac{1}{2\mathrm{Vol}_{\grf}(M)}\int_M\langle h,\Lap^L_{\grf} h\rangle_{\grf} dV_{\grf} & h\in\ker\div_{\grf}, \\
0 & h\in\mathrm{im}\,\div_{\grf}^\ast.
\end{array} \right.
\end{equation}
This can also be computed using (\ref{pertseries2}) (see Appendix \ref{appendixsecondvariation}).

\begin{proposition}[Third variation of $\lambda$]\label{3rdvaroflambda}
Let $(M,g_0)$ be a compact Riemannian manifold. Then there exists a $C^{2,\alpha}$-neighborhood\, $\mathcal{U}_{g_0}\subset\mathcal{M}(M)$ of $g_0$ in the space of metrics on $M$ and a constant $C< \infty$ such that
\begin{equation}
\left|\tfrac{d^3}{d\eps^3}|_0\lambda(g+\eps h)\right|\leq C \norm{h}_{C^{2,\alpha}} \norm{h}_{H^1}^2
\end{equation}
for all $g\in\mathcal{U}_{g_0}$ and all $h\in C^\infty(S^2T^\ast M)$.
\end{proposition}

\begin{proof}
We have $C^{2,\alpha}$-bounds for the ground-state $w_g$ of $H_g$, which we will often use in the following. Let us estimate (\ref{pertseries3}) term by term. The first term has the schematic form
\begin{align}
&\langle w, H'''[h,h,h]w\rangle\\
&\qquad=\langle w, (\Rm h h h + h h D h D + h D h D h+h h h D^2 + h h D^2h) w\rangle.\nonumber
\end{align}
Since $M$ is compact, we get the estimate
\begin{equation}
|\langle w_g, H_g'''[h,h,h]w_g\rangle |\leq C \norm{h}_{C^2} \norm{h}_{H^1}^2.
\end{equation}
Let us continue with the second term,
\begin{equation}\label{secondterm}
\oint_{|\lambda-\lambda(g)|=r}\langle w_g,H_g'[h](\lambda-H_g)^{-1}H_g'[h](\lambda-H_g)^{-1}H_g'[h] w_g\rangle\tfrac{d\lambda}{\lambda-\lambda(g)}.
\end{equation}
Recall that for $|\lambda-\lambda(g)|=r$ the operator $\lambda-H_g:C^{\infty}(M)\to C^{\infty}(M)$
is indeed invertible and that $H_g$ is symmetric with respect to the $L^2(M,dV_g)$-inner product. Let us insert the left and right
\begin{equation}\label{hprime2}
H'[h]=\Rc\, h +D h D + h D^2  +D^2 h
\end{equation}
in (\ref{secondterm}). By partial integration, it can be brought into the form
\begin{equation}
\oint_{|\lambda-\lambda(g)|=r}\langle v_{\bar{\lambda}}[h],H'[h]v_\lambda[h] \rangle\tfrac{d\lambda}{\lambda-\lambda(g)},
\end{equation}
where
\begin{equation}
v_\lambda[h]=(\lambda-H_g)^{-1}\left( \Rc\, wh+ D w D h+D^2 w h  + w D^2h\right).
\end{equation}
We have the elliptic estimate
\begin{equation}\label{pertest}
\norm{v_\lambda[h]}_{L^2}\leq C \norm{h}_{L^2}.
\end{equation}
Indeed $(\lambda-H_g)^{-1}\!:H^{-2}(M)\to L^2(M)$ is well defined and continuous, since it is the dual of the continuous map $(\lambda-H_g)^{-1}\!:L^2(M)\to H^2(M)$. The constant in (\ref{pertest}) can be chosen uniformly for all $\lambda$ on the circle around $\lambda(g)$, since we have a lower bound for the distance between this circle and the spectrum of $H_g$.
Finally, we insert the middle $H'[h]$ and take care of $hD^2$ by partial integration. Putting everything together, we obtain
\begin{equation}
\left| \oint \langle w,H'[h](\lambda-H)^{-1}H'[h](\lambda-H)^{-1}H'[h] w\rangle\tfrac{d\lambda}{\lambda-\lambda(g)} \right| \leq C \norm{h}_{C^2}\norm{h}_{H^1}^2.
\end{equation}
To continue, inserting (\ref{hprime2}) and
\begin{equation}
H''[h,h]=\Rm hh+hhD^2+hD hD+hD^2h+D hD h
\end{equation}
and using partial integration, the third and the fourth term can be brought into the form
\begin{equation}
\oint \langle \Rm whh+D^2whh+D whD h + whD^2 h +whD h D,v_{\lambda\text{ or }\bar{\lambda}}[h]\rangle\tfrac{d\lambda}{\lambda-\lambda(g)}.
\end{equation}
With the $L^2$-estimates, this can be bounded by $C\norm{h}_{C^2}\norm{h}_{H^1}^2$.\\
For the last term, note that
\begin{equation}
|\langle w,H'[h]w\rangle|\leq C \norm{h}_{C^2}.
\end{equation}
Inserting $H'[h]$ and taking care of $h D^2+D^2 h$ by partial integration the last integral can be estimated by
\begin{equation}
\left|\oint \langle w,H'[h]v_\lambda[h]\rangle \tfrac{d\lambda}{(\lambda-\lambda(g))^2}\right|\leq C \norm{h}_{H^1}^2.
\end{equation}
Finally, by Lemma \ref{boundonw} below, we have
$\norm{w_g}_{C^{2,\alpha}} \leq C$ uniformly for all $g$ in a $C^{2,\alpha}$-neighborhood $\mathcal{U}_{g_0}\subset\mathcal{M}(M)$ of $g_0$. This uniform bound and the continuity of eigenvalues discussed at the beginning of Section \ref{variationalstructure} show that all the above estimates go through uniformly in a small enough neighborhood $\mathcal{U}_{g_0}$ of $g_0$.
\end{proof}

\begin{lemma}\label{boundonw}
Let $(M,g_0)$ be a compact Riemannian manifold. Then there exists a $C^{2,\alpha}$-neighborhood\, $\mathcal{U}_{g_0}\subset\mathcal{M}(M)$ of $g_0$ and a constant $C<\infty$ such that
\begin{equation}
\norm{w_g}_{C^{2,\alpha}}\leq C
\end{equation}
for all $g\in\mathcal{U}_{g_0}$, where $w_g$ denotes the ground state of $H_g=-4\Lap_g+R_g$.
\end{lemma}

\begin{proof}
By definition of the ground state,
\begin{equation}
(-4\Lap_g+R_g-\lambda(g))w_g=0,\qquad \norm{w_g}_{L^2(M,dV_g)}=1.
\end{equation}
By DeGiorgi-Nash-Moser and Schauder estimates \cite[Thm. 8.17, Thm. 6.2]{GT}
\begin{equation}
\norm{w_g}_{C^{2,\alpha}}\leq C\norm{w_g}_{L^2}\leq C.
\end{equation}
Here, for definiteness, we define the norms using the background metric $g_0$. The constants $\lambda(g)$ are uniformly bounded by the continuity of eigenvalues and we also have a uniform $C^{0,\alpha}$-bound for the coefficient $R_g$ and good control over $\Lap_g$. Thus, the estimates are uniform in a $C^{2,\alpha}$-neighborhood of $g_0$.
\end{proof}

\section{The Ebin-Palais slice theorem and integrability}\label{ebinandintegrability}

Let us recall some facts from \cite{Eb}. Fix a compact manifold $M$. The group of diffeomorphisms $\mathcal{D}(M)$ acts on the space of metrics $\mathcal{M}(M)\subset C^\infty(S^2T^\ast M)$ by pullback. Fix $g_0\in\mathcal{M}(M)$. Since $\div^\ast_{g_0}$ is overdetermined elliptic we have the $L^2$-orthogonal decomposition:
\begin{equation}\label{infinitesimaldecomp}
C^\infty(S^2T^\ast M)=\ker\div_{g_0}\oplus\text{im}\,\div^\ast_{g_0}.
\end{equation}
Let $\mathcal{O}_{g_0}\subset\mathcal{M}(M)$ be the orbit of $g_0$ under the action of $\mathcal{D}(M)$. By the Ebin-Palais slice theorem, there exists a slice $\mathcal{S}_{g_0}$ for the action of $\mathcal{D}(M)$ on $\mathcal{M}(M)$. In particular, $\mathcal{M}(M)\cong_{\mathrm{loc}} \mathcal{S}_{g_0}\times \mathcal{O}_{g_0}$ is locally a product near $g_0$ (in the sense of inverse limit Banach manifolds) and the induced decomposition of $T_{g_0}\mathcal{M}(M)=C^\infty(S^2T^\ast M)$ is given by (\ref{infinitesimaldecomp}). We will only use the following part of the theorem.

\begin{theorem}[{Ebin-Palais \cite{Eb}}]\label{ebinslicethm}
Let M be a compact manifold and $\mathcal{M}(M)$ the space of metrics on $M$.\\
Then for every metric $g_0\in\mathcal{M}(M)$, there exists a $C^{2,\alpha}$-neighborhood $\mathcal{U}_{g_0}\subset\mathcal{M}(M)$ of $g_0$, such that every metric $g\in\mathcal{U}_{g_0}$ can be written as $g=\varphi^\ast \hat{g}$ for some diffeomorphism $\varphi\in\mathcal{D}(M)$ and some metric $\hat{g}\in\mathcal{S}_{g_0}= (g_0+\ker\div_{g_0})\cap \mathcal{U}_{g_0}$.
\end{theorem}

\begin{remark}
Ebin uses Sobolev spaces, Palais uses H\"{o}lder spaces. Moreover, Ebin uses the exponential map $\text{Exp}_g$ of the $L^2$-metric on $\mathcal{M}(M)$ to construct his slice. Palais uses the map $E_g(h)=g+h$ and thus gets an \emph{affine} slice $\mathcal{S}_{g_0}\subset (g_0+\ker\div_{g_0}) \cap \mathcal{M}(M)$ (the crucial property for constructing the slice is that the exponential map is equivariant, i.e. $\varphi^\ast(\text{Exp}_gh)$=$\text{Exp}_{\varphi^\ast g}\varphi^\ast h$, which is true for the `exponential map' $E$, since the action is linear).
\end{remark}

Let us now discuss our integrability assumption, following \cite[Sec. 12]{Be}.
\begin{definition}
Let $M$ be compact and $\grf\in\mathcal{M}(M)$ Ricci-flat. We call
\begin{equation}
I_{\grf}=\{ h\in \ker\div_{\grf};\; D\Rc(\grf)[h]=0\}
\end{equation}
the space of infinitesimal Ricci-flat deformations of $\grf$ and
\begin{equation}
\mathcal{P}_{\grf}=\{ \bar{g}\in\mathcal{S}_{\grf};\; \Rc_{\bar{g}}=0\}
\end{equation}
the premoduli space of Ricci-flat metrics near $\grf$ (the true moduli space is modeled on $\mathcal{P}_{\grf}/\mathrm{Isom}_{\grf}$).
\end{definition}

\begin{lemma}\label{integrability1}
Let $(M,\grf)$ compact, Ricci-flat. Then $I_{\grf}=\mathbb{R}\grf\oplus K_{\grf}$, where
\begin{equation}
K_{\grf}=\{h\in C^\infty(S^2T^\ast M);\; \div_{\grf} h=0,\, \mathrm{tr}_{\grf} h=0,\, \Lap^L_{\grf} h=0\}.
\end{equation}
\end{lemma}
\begin{proof}
On transverse symmetric 2-tensors, the linearization of Ricci is proportional to $\Lap^L +D^2\circ\text{tr}$.
Thus for $h\in I_{\grf}$, we have $\Lap^Lh +D^2 \text{tr }h=0$. Taking the trace, we get $\Lap \text{tr }h=0$, thus $\text{tr }h=c$ and $\Lap^Lh=0$. Therefore
\begin{equation}
h=\tfrac{c}{n}\grf+(h-\tfrac{c}{n}\grf)\in \mathbb{R}\grf\oplus K_{\grf}.
\end{equation}
The converse inclusion is clear.
\end{proof}

\begin{definition}[Integrability]
Let $M$ be compact and $\grf\in\mathcal{M}(M)$ be Ricci-flat. We say that all infinitesimal Ricci-flat deformations of $\grf$ are integrable if there is a smooth familiy $g_h(t)\in\mathcal{M}(M)$ of Ricci-flat metrics with $g_h(0)=\grf$ and $\dot{g}_h(0)=h$, defined for all $h\in I_{\grf}$ with norm less then one and all $t\in (-\eps,\eps )$.
\end{definition}

\begin{proposition}\label{integrability2}
Let $M$ be compact and $\grf\in\mathcal{M}(M)$ be Ricci-flat. If all infinitesimal Ricci-flat deformations of $\grf$ are integrable, then $\mathcal{P}_{\grf}$ is a manifold near $\grf$ with $T_{\grf}\mathcal{P}_{\grf}=I_{\grf}$.
\end{proposition}

\begin{proof}
As in the proof of Koiso's theorem, we construct a manifold $\mathcal{Z}_{\grf}\subset\mathcal{S}_{\grf}$ near $\grf$ that contains $\mathcal{P}_{\grf}$ and satisfies $T_{\grf}\mathcal{Z}_{\grf}=I_{\grf}$. Possibly after passing to smaller neighborhoods, we have $\mathcal{P}_{\grf}=\mathcal{Z}_{\grf}$ due to integrability (see \cite[Thm. 12.49]{Be} for details).
\end{proof}

\section{Local maxima, gradient inequality and transversality}\label{maintheorems}

In this section, we prove A, B, C and D.

\begin{proof}[Proof of Theorem A]
Let $\mathcal{U}_{\grf}\supset\mathcal{S}_{\grf}\supset\mathcal{P}_{\grf}$ be as in Section \ref{ebinandintegrability}. We divide the proof of the theorem into the following three steps, whose detailed proofs can be found below:
\begin{enumerate}[i.]
\item\label{stepi} For $\ker\div_{\grf} = T_{\grf}\mathcal{P}_{\grf}\oplus N_{\grf}$, where
\begin{equation}
N_{\grf}=\{h\in\ker\div_{\grf};\; \langle h,k\rangle_{L^2_{\grf}}=0\;\text{for all}\; k\in T_{\grf}\mathcal{P}_{\grf}\},
\end{equation}
the second variation $D^2\lambda(\grf)$ vanishes on the first summand and is strictly negative on the second one.
\item\label{stepii} By Taylor expansion with careful estimate of the error term, possibly after passing to smaller neighborhoods, $\lambda$ is nonpositive on $\mathcal{S}_{\grf}$ and vanishes only on $\mathcal{P}_{\grf}$.
\item\label{stepiii} The assertion of the theorem follows from the Ebin-Palais slice theorem and the diffeomorphism invariance of $\lambda$.
\end{enumerate}
\begin{proof}[Proof of \ref{stepi}]
Since $\grf$ is Ricci-flat, we have the $L^2$-orthogonal, $\Lap^L_{\grf}$-invariant decomposition \cite[Sec. 4]{GIK},
\begin{equation}
\ker\div_{\grf}=\mathbb{R}\grf\oplus \mathrm{im}(C_{\grf}) \oplus TT_{\grf},
\end{equation}
where $\mathbb{R}\grf$ describes scaling, $C_{\grf}u=(\Lap_{\grf} u)\grf-\Hess_{\grf}u$ describes the other conformal transformations (projected on $\ker\div_{\grf}$) and
\begin{equation}
TT_{\grf}=\{h\in C^\infty(S^2T^\ast M);\; \div_{\grf} h=0, \text{tr}_{\grf} h = 0\}
\end{equation}
denotes the space of transverse, traceless, symmetric 2-tensors.\\
Let us analyse the spectrum. The Lichnerowicz Laplacian $\Lap^L_{\grf}$ vanishes on $\mathbb{R}\grf$. It is strictly negative on $\text{im}(C)$, since $\Lap^L Cu=C\Lap u$. Indeed, taking the trace shows that the elements of the kernel of $C$ are harmonic and thus constant functions (the theorem is trivial in one dimension, where every metric is flat and $\lambda$ vanishes identically). So, given the eigenvalue equation,
\begin{equation}
\Lap^L Cu = \alpha Cu,\qquad Cu\neq 0,
\end{equation}
by adding a constant, we can assume without loss of generality $\int_M u=0$. Now
\begin{equation}
C(\Lap u-\alpha u)=\Lap^LCu-\alpha Cu=0,
\end{equation}
so $\Lap u -\alpha u$ is constant and by integration this constant is seen to be zero. Thus $\alpha\leq 0$. If $\alpha$ were zero, then $u$ would be constant and $Cu=0$, a contradiction. Finally, $\Lap^L$ is nonpositive on $TT$ by the hypothesis of the theorem (more precisely, by the weaker hypothesis $\Lap^L\leq 0$ on TT). The kernel
\begin{equation}
K_{\grf}=\{h\in TT_{\grf} ;\, \Lap^L_{\grf} h = 0\}
\end{equation}
is finite dimensional and $\Lap^L$ is strictly negative on $TT_{\grf}\ominus K_{\grf}$.\\
By Lemma \ref{integrability1} and Proposition \ref{integrability2}, $T_{\grf}\mathcal{P}_{\grf}=\mathbb{R}\grf\oplus K_{\grf}$. Now claim \ref{stepi}.\! follows from (\ref{2ndvaroflambda}). More precisely, there exists a constant $c>0$, such that
\begin{equation}\label{eigenestimate}
\langle h, \Lap^L_{\grf} h\rangle_{L^2_{\grf}} \leq - c\langle h, h\rangle_{L^2_{\grf}} \qquad \text{ for all } h\in N_{\grf}.
\end{equation}
\end{proof}

\begin{proof}[Proof of \ref{stepii}]
For small $\eps>0$, by continuity,
\begin{align}
\langle h, \Lap^L_{\bar{g}} h\rangle_{L^2_{\bar{g}}}&=-\eps \langle D h, D h\rangle_{L^2_{\bar{g}}}+(1-\eps)\langle h, \Lap_{\bar{g}} h+\tfrac{2}{1-\eps} \Rm_{\bar{g}}\!:\!h\rangle_{L^2_{\bar{g}}}\nonumber\\
& \leq  - c \norm{h}^2_{H^1} \qquad \text{ for all } \bar{g}\in\mathcal{P}_{\grf},\,h\in N_{\grf}
\end{align}
for some new constant $c>0$, possibly after passing to smaller neighborhoods. Now $\bar{g}\in\mathcal{P}_{\grf}$ is Ricci-flat, so $\lambda(\bar{g})=0$ and $D\lambda({\bar{g}})=0$. Thus
\begin{equation}\label{taylorlambda}
\lambda({\bar{g}+h})\leq -c \norm{h}^2_{H^1} + |R(\bar{g},h)|.
\end{equation}
Here we used the formula
\begin{align}
\lambda(\bar{g}+h)&=\lambda(\bar{g})+\tfrac{d}{dt}|_0\lambda(\bar{g}+th)+\tfrac{1}{2}\tfrac{d^2}{dt^2}|_0\lambda(\bar{g}+th)+R(\bar{g},h),\\
R(\bar{g},h)&=\int_0^1\left(\tfrac{1}{2}-t+\tfrac{1}{2}t^2\right)\tfrac{d^3}{dt^3}\lambda(\bar{g}+th)dt.
\end{align}
By Proposition \ref{3rdvaroflambda} we have the uniform estimate
\begin{equation}
|R(\bar{g},h)|\leq C  \norm{h}_{C^{2,\alpha}}  \norm{h}^2_{H^1}
\end{equation}
for the remainder, if $\bar{g}-\grf$ and $h$ are $C^{2,\alpha}$-small. For sufficiently small $C^{2,\alpha}$-norm, the negative term in (\ref{taylorlambda}) dominates.
Finally, the `exponential map'
\begin{equation}
E:\mathcal{P}_{\grf}\times N_{\grf}\to \grf+\ker\div_{\grf},\qquad E(\bar{g},h)=\bar{g}+h
\end{equation}
maps a $C^{2,\alpha}$-neighborhood of $(\grf,0)$ onto a $C^{2,\alpha}$-neighborhood of $\grf$. Here, to apply the inverse function theorem, we temporarily enlarge the involved spaces to $C^{2,\alpha}$-spaces. Since the kernel of $\Lap^L_{\grf}$ is smooth by elliptic regularity, the proof of Proposition \ref{integrability2} shows that $\mathcal{P}_{\grf}$ only consists of smooth elements also after passing to $C^{2,\alpha}$-spaces. Thus $E(\bar{g},h)$ is smooth if and only if $h$ is smooth. This finishes the proof of Claim \ref{stepii}.
\end{proof}

\begin{proof}[Proof of \ref{stepiii}]
By the Ebin-Palais slice theorem, every $g\in\mathcal{U}_{\grf}$ can be written as $g=\varphi^\ast \hat{g}$ for some $\varphi\in\mathcal{D}(M), \hat{g}\in\mathcal{S}_{\grf}$. Since $\lambda$ is diffeomorphism invariant
\begin{equation}
\lambda(g)=\lambda(\hat{g})\leq 0
\end{equation}
by step \ref{stepii}. If $\lambda(g)=0$, then $\hat{g}\in\mathcal{P}_{\grf}$, so $\Rc_{\hat{g}}=0$ and thus $\Rc_g=0$.
\end{proof}

This finishes the proof of Theorem A.
\end{proof}

\begin{proof}[Proof of Corollary D]
Let $\mathcal{U}_{\grf}$ be as in Theorem A and $g\in\mathcal{U}_{\grf}$. If $R_g\geq 0$, then $\lambda(g)\geq 0$. Thus $\lambda(g)=0$ and $\Rc_g=0$ by the equality case of Theorem A.
\end{proof}

We will now estimate the motion in the gauge directions. Namely, we have to deal with the minimizer $f_g$ from (\ref{lambda}) appearing in $e^{-f_g}dV_g$ and more importantly in $\Rc_g+\Hess_g f_g$ in (\ref{1stvaroflambda}). We start with the following refinement of Lemma \ref{boundonw}.

\begin{lemma}\label{ctwoalphalemma}
Let $(M,\grf)$ be compact, Ricci-flat and $\eps>0$. Then there exists a $C^{2,\alpha}$-neighborhood $\mathcal{U}_{\grf}$ of $\grf$ such that
\begin{equation}
\norm{f_g-\log\mathrm{Vol}_{\grf}(M)} _{C^{2,\alpha}}<\eps
\end{equation}
for all $g\in\mathcal{U}_{\grf}$, where $f_g$ is the minimizer in (\ref{lambda}).
\end{lemma}

\begin{proof}
Assume the volume is normalized, then
\begin{equation}
f_{\grf}=\log\mathrm{Vol}_{\grf}(M)=0.
\end{equation}
Write $w_g=e^{-f_g/2}$. There is some $\tilde{\eps}>0$, such that
\begin{equation}\label{eqntoshow}
\norm{w_g-1}_{C^{2,\alpha}}<\tilde{\eps}\Rightarrow \norm{f_g}_{C^{2,\alpha}}<\eps
\end{equation}
We will prove $\norm{w_g-1}_{C^{2,\alpha}}<\tilde{\eps}$ for $g$ near $\grf$ using the implicit function theorem. Let
\begin{align}
X&=\{g\in C^{2,\alpha}(S^2T^\ast M);\; g \text{ positive definite}\},\\
Y&=\{u\in C^{2,\alpha}(M);\;  \int_M u\, dV_{\grf}=0\},\\
Z&=\{l\in C^{0,\alpha}(M);\; \int_M l\, dV_{\grf}=0\}.
\end{align}
Define $F:X\times Y\to Z$ by
\begin{multline}
F(g,u)=(-4\Lap_g+R_g-\lambda(g))(1+u)\\
-\int_M (-4\Lap_g+R_g-\lambda(g))(1+u)\, dV_{\grf}.
\end{multline}
From Section \ref{variationalstructure}, we know that $F$ is $C^1$.
Observe that $F(\grf,0)=0$ and
\begin{equation}\label{solveforeigenvalue}
F(g,u)=0\Leftrightarrow (-4\Lap_g+R_g)(1+u)=\lambda(g)(1+u).
\end{equation}
Indeed, $F(g,u)=0$ implies $(-4\Lap_g+R_g-\lambda(g))(1+u)=c$ and by the Fredholm alternative $\int_Mcw_gdV_g=0$. Thus $c=0$, since $w_g$ is positive. Now
\begin{equation}
DF(\grf,0)|_Y=-4\Lap_{\grf}:Y\to Z
\end{equation}
is indeed an isomorphism. By the implicit function theorem there exists a $C^{2,\alpha}$-neighborhood of $\grf$ such that (\ref{solveforeigenvalue}) can be solved for $u=u(g)$ with the estimate $\norm{u(g)}_{C^{2,\alpha}}<\tilde{\eps}/100$. Since
\begin{equation}
w_g=\left(\int_M(1+u(g))^2dV_g\right)^{-\tfrac{1}{2}}(1+u(g))
\end{equation}
we obtain $\norm{w_g-1}_{C^{2,\alpha}}<\tilde{\eps}$ in a small enough $C^{2,\alpha}$-neighborhood.
\end{proof}

Let $g\in \grf+\ker\div_{\grf}, g=\bar{g}+h, \bar{g}\in\mathcal{P}_{ \grf}, h\in N_{ \grf}$ as in the proof of Theorem A. In the following four lemmas, we will show
\begin{equation}
\Rc_g+\Hess_g f_g=-\tfrac{1}{2}\Lap^L_{\grf}h+O_1(h\ast h)+O_2((\bar{g}-\grf)\ast h)
\end{equation}
in a $C^{2,\alpha}$-neighborhood of $\grf$ with estimates for $O_1$ and $O_2$.

\begin{lemma}
Let $(M,\grf)$ be compact Ricci-flat and $h\in\ker\div_{\grf}$. Then
\begin{align}
&\tfrac{d}{dt}|_0 f_{\grf+th}=\tfrac{1}{2}\tr_{\grf}h\label{dtf}\\
&\tfrac{d}{dt}|_0\left( \Rc_{\grf+th}+\Hess_{\grf+th} f_{\grf+th}\right)=-\tfrac{1}{2}\Lap^L_{\grf}h.\label{dts}
\end{align}
\end{lemma}

\begin{proof}
Since $f_{\grf}=\log\mathrm{Vol}_{\grf}(M)$ is a constant function, many terms will drop out in the following computation. From Section \ref{variationalstructure}, we know that $t\mapsto  f_{\grf+th}$ is analytic. Differentiating the equations
\begin{align}
&\left(-4\Lap_{\grf+th}+R_{\grf+th}-\lambda({\grf+th})\right)e^{-\tfrac{1}{2}f_{\grf+th}}=0,\\
&\int_M e^{-f_{\grf+th}} dV_{\grf+th}=1
\end{align}
at $t=0$, we obtain
\begin{align}
&\Lap_{\grf}\left(\tfrac{d}{dt}|_0 f_{\grf+th}-\tfrac{1}{2}\tr_{\grf}h\right)=0,\\
&\int_M\left(\tfrac{d}{dt}|_0 f_{\grf+th}-\tfrac{1}{2}\tr_{\grf}h\right)dV_{\grf}=0,
\end{align}
and Equation (\ref{dtf}) follows. Equation (\ref{dts}) follows from
\begin{equation}
\tfrac{d}{dt}|_0 \Rc_{\grf+th}=-\tfrac{1}{2}\left(\Lap^L_{\grf}h+\Hess_{\grf}\tr_{\grf}h\right)
\end{equation}
and $\tfrac{d}{dt}|_0\left( \Hess_{\grf+th} f_{\grf+th}\right)=\Hess_{\grf}\tfrac{d}{dt}|_0 f_{\grf+th}$.
\end{proof}

\begin{lemma}\label{taylor}
Let $F(s,t)$ be a $C^2$-function on $0\leq s,t\leq 1$ with values in a Frechet-space. Then
\begin{multline}
F(1,1)=F(1,0)+\tfrac{d}{dt}|_0F(0,t)+\int_0^1(1-t)\tfrac{d^2}{dt^2}F(0,t)dt\\
+\int_0^1\int_0^1\tfrac{\partial^2}{\partial s\partial t}F(s,t)dsdt.
\end{multline}
\end{lemma}

\begin{proof}
By the Hahn-Banach theorem, it suffices to prove the lemma for real valued F and this follows from
\begin{align}
\int_0^1(1-t)\tfrac{d^2}{dt^2}F(0,t)dt&=-\tfrac{d}{dt}|_0F(0,t)+\underbrace{\int_0^1\tfrac{d}{dt}F(0,t)dt}_{=F(0,1)-F(0,0)},\\
\int_0^1\int_0^1\tfrac{\partial^2}{\partial s\partial t}F(s,t)dsdt&=F(1,1)+F(0,0)-F(1,0)-F(0,1).
\end{align}
\end{proof}

\begin{lemma}\label{taylorsteady}
Let $g\in \grf+\ker\div_{\grf}, g=\bar{g}+h, \bar{g}\in\mathcal{P}_{ \grf}, h\in N_{ \grf}$ as in the proof of Theorem A. Then, in a $C^{2,\alpha}$-neighborhood of $\grf$ in $\grf+\ker\div_{\grf}$, we have the equality
\begin{equation}
\Rc_g+\Hess_g f_g=-\tfrac{1}{2}\Lap^L_{\grf}h+O_1+O_2
\end{equation}
with
\begin{align}
O_1&=\int_0^1(1-t)\tfrac{d^2}{dt^2}\left( \Rc_{\grf+th}+\Hess_{\grf+th} f_{\grf+th}\right)dt,\\
O_2&=\int_0^1\int_0^1\tfrac{\partial^2}{\partial s\partial t}\Big( \Rc_{\grf+s(\bar{g}-\grf)+th}\nonumber\\
&\qquad\qquad\qquad+\Hess_{\grf +s(\bar{g}-\grf) +th} f_{\grf +s(\bar{g}-\grf) +th}\Big)dsdt.
\end{align}
\end{lemma}

\begin{proof}
Use Lemma \ref{taylor} with
\begin{equation}
F(s,t)=\Rc_{\grf+s(\bar{g}-\grf)+th}+\Hess_{\grf +s(\bar{g}-\grf) +th} f_{\grf +s(\bar{g}-\grf) +th}.
\end{equation}
Note that $F(1,0)=\Rc_{\bar{g}}+\Hess_{\bar{g}} f_{\bar{g}}=0$ and use (\ref{dts}).
\end{proof}

\begin{lemma}\label{estimateofo}
Let $g\in \grf+\ker\div_{\grf}, g=\bar{g}+h, \bar{g}\in\mathcal{P}_{ \grf}, h\in N_{ \grf}$ as in the proof of Theorem A. Then, there exists a $C^{2,\alpha}$-neighborhood of $\grf$ in $\grf+\ker\div_{\grf}$ and a constant $C<\infty$ such that the inequalities
\begin{align}
\norm{O_1}_{L^2}&\leq C \norm{h}_{C^{2,\alpha}} \norm{h}_{H^2},\label{technicalestimate1}\\
\norm{O_2}_{L^2}&\leq C \norm{\bar{g}-\grf}_{C^{2,\alpha}} \norm{h}_{H^2}\label{technicalestimate2}
\end{align}
hold in this neighborhood.
\end{lemma}

\begin{proof}
The estimate is clear for the part of $O_i$ coming from $\Rc$ (since it contains at most second derivatives). The part coming from $\Hess f$ is more tricky. Let $h,k$ be symmetric 2-tensors. We will show
\begin{equation}
\norm{\tfrac{\partial^2}{\partial s\partial t}|_{(0,0)} \Hess_{g+sk+th} f_{g+sk+th}}_{L^2} \leq C \norm{k}_{C^{2,\alpha}} \norm{h}_{H^2}
\end{equation}
uniformly for all $g$ in a $C^{2,\alpha}$-neighborhood. We differentiate:
\begin{equation}
\tfrac{\partial^2}{\partial s\partial t}|_{(0,0)} \Hess_{g+sk+th} f_{g+sk+th}=\dot{\Hess}'f+\Hess'\dot{f}+\dot{\Hess}f'+\Hess\dot{f}'.
\end{equation}
The first term has the schematic form
\begin{equation}
\dot{\Hess}'f=kDhDf+hDkDf,
\end{equation}
thus
\begin{equation}
\norm{\dot{\Hess}'f}_{L^2}\leq C \norm{k}_{C^1} \norm{h}_{H^1}\leq C  \norm{k}_{C^{2,\alpha}} \norm{h}_{H^2}
\end{equation}
by Lemma \ref{ctwoalphalemma} (we will use Lemma \ref{ctwoalphalemma} frequently below without mentioning it again). To control the other three terms, we will differentiate the equation
\begin{equation}
2\Lap_{g+sk+t h}f_{g+sk+t h}-|Df_{g+sk+t h}|_{g+sk+t h}^2+R_{g+sk+t h}=\lambda(g+sk+t h)\label{technical1}
\end{equation}
and use elliptic estimates. Differentiating with respect to $t$ gives the linear elliptic equation
\begin{equation}
P_g\dot{f}=F_g[h],
\end{equation}
where $P_g=\Lap_g-g^{ij}D_if_gD_j$ and $F$ has the schematic form
\begin{equation}
F_g[h]=\dot{\lambda}+D^2h+DhDf+hD^2f+hDfDf+h\Rc.
\end{equation}
By the maximum principle, only constant functions are in the kernel of $P$. Thus $\dot{f}-\bar{\dot{f}}$ is $L^2$-orthogonal to $\ker P$ (the bar denotes the average). Since it also solves the equation
\begin{equation}
P_g(\dot{f}-\bar{\dot{f}})=F_g[h],
\end{equation}
we get the estimate
\begin{equation}\label{technical2}
\norm{\dot{f}-\bar{\dot{f}}}_{H^2}\leq C \norm{F_g[h]}_{L^2}\leq C \norm{h}_{H^2}.
\end{equation}
In the last step, we used the estimate (cf. Section \ref{variationalstructure}),
\begin{equation}
|\dot{\lambda}|\leq C \norm{h}_{H^2}.
\end{equation}
Thus
\begin{equation}
\norm{\Hess'\dot{f}}_{L^2}\leq C  \norm{DkD\dot{f}}_{L^2}\leq C  \norm{k}_{C^1} \norm{\dot{f}-\bar{\dot{f}}}_{H^1}
\leq C \norm{k}_{C^{2,\alpha}} \norm{h}_{H^2}.
\end{equation}
Next, we will estimate $\dot{\Hess}f'$. Similar as above, we obtain:
\begin{align}
P_g(f'-\bar{f'})&=F_g[k]\label{technical3}\\
\norm{f'-\bar{f'}}_{H^2}&\leq C \norm{k}_{H^2}.\label{technical4}
\end{align}
From (\ref{technical3}), by DeGiorgi-Nash-Moser and Schauder estimates we get
\begin{equation}
\norm{f'-\bar{f'}}_{C^{2,\alpha}}\leq C \left( \norm{F_g[k]}_{C^{0,\alpha}}+\norm{f'-\bar{f'}}_{L^2}\right)
\leq C \norm{k}_{C^{2,\alpha}},\label{technical5}
\end{equation}
where we used (\ref{technical4}) and $|\lambda'|\leq C\norm{k}_{C^{2,\alpha}}$. Thus
\begin{equation}
\norm{\dot{\Hess}f'}_{L^2}\leq C  \norm{DhDf'}_{L^2}\leq C \norm{k}_{C^{2,\alpha}} \norm{h}_{H^2}.
\end{equation}
Finally, let us estimate $\Hess\dot{f}'$. Differentiating (\ref{technical1}) twice gives the linear elliptic equation
\begin{equation}
P_g\dot{f}'=G_g[h,k],
\end{equation}
where $G$ has the schematic form,
\begin{align}
G_g[h,k]&=\dot{\lambda}'+D\dot{f}Df'+\left(hD^2f'+DhDf'+hDf'Df\right)\nonumber\\
&\quad+\left(kD^2\dot{f}+DkD\dot{f}+kD\dot{f}Df\right)\nonumber\\
&\quad+\Big(kD^2h+hD^2k+DhDk+hDkDf+kDhDf\nonumber\\
&\qquad\qquad\qquad\qquad+hkDfDf+hkD^2f+hk\Rm\Big).\label{technical6}
\end{align}
Similar as before, we get the estimate
\begin{equation}
\norm{\Hess\dot{f}'}_{H^2}\leq C\norm{\dot{f}'-\bar{\dot{f}}'}_{H^2}\leq C \norm{G_g[h,k]}_{L^2}\leq C \norm{k}_{C^{2,\alpha}} \norm{h}_{H^2},
\end{equation}
where the last inequality is obtained as follows: The expression (\ref{technical6}) for $G$ consists of five terms. The inequality is clear for the fifth term, for the fourth term it follows from (\ref{technical2}), for the third term from (\ref{technical5}) and for the second term from (\ref{technical2}) and (\ref{technical5}). Finally, from Section \ref{variationalstructure}, we know
\begin{equation}
|\dot{\lambda}'|\leq C \norm{k}_{C^{2,\alpha}} \norm{h}_{H^2},
\end{equation}
and this yields the inequality for the first term. Indeed, from (\ref{pertseries2}) by polarization
\begin{align}
\tfrac{\partial^2}{\partial s\partial t}|_{(0,0)}\lambda(g+sk+t h)=&\langle w, \dot{H}'[h,k]w\rangle\\
&+\tfrac{1}{2\pi i}\oint\langle w,\dot{H}[h](\lambda-H)^{-1}H'[k] w\rangle\tfrac{d\lambda}{\lambda-\lambda(g)}\nonumber\\
&+\tfrac{1}{2\pi i}\oint\langle w,H'[k](\lambda-H)^{-1}\dot{H}[h] w\rangle\tfrac{d\lambda}{\lambda-\lambda(g)}\nonumber
\end{align}
and this can be estimated using the same methods as in the proof of Proposition \ref{3rdvaroflambda}. All the above estimates are uniform in a $C^{2,\alpha}$-neighborhood. This finishes the proof of the lemma.
\end{proof}

\begin{proof}[Proof of Theorem C]
We can assume  $g\in \grf+\ker\div_{\grf}, g=\bar{g}+h, \bar{g}\in\mathcal{P}_{ \grf}, h\in N_{ \grf}$. This reduction is justified using the Ebin-Palais slice theorem and integrability as in the proof of Theorem A. In particular, note that $\varphi^\ast f_g=f_{\varphi^\ast g}$ and that the different $L^2$-norms are uniformly equivalent.\\
Since
\begin{equation}
\norm{\Rc_g}_{L^2}\leq C \norm{h}_{H^2},
\end{equation}
it suffices to show
\begin{equation}\label{toshow}
\norm{\Rc_g+\Hess_gf_g}_{L^2}^2\geq c  \norm{h}_{H^2}^2
\end{equation}
for some $c>0$. To see this, using Lemma \ref{taylorsteady}, note that
\begin{align}
\norm{\Rc_g+\Hess_gf_g}_{L^2}^2&=\tfrac{1}{4}\norm{\Lap^L_{\grf}h}_{L^2}^2-\langle O_1+O_2, \Lap^L_{\grf}h\rangle+\norm{O_1+O_2}_{L^2}^2\nonumber\\
&\geq 2c\norm{h}_{H^2}^2-C(\norm{O_1}_{L^2}+\norm{O_2}_{L^2}) \norm{h}_{H^2}
\end{align}
for some $c>0$, since $\Lap^L_{\grf}|_{N_{ \grf}}$ is injective. Together with Lemma \ref{estimateofo}, this proves (\ref{toshow}) in a $C^{2,\alpha}$-neighborhood and the theorem follows.
\end{proof}

\begin{remark}\label{reverseremark}
The reverse inequality,
\begin{equation}
 \norm{\Rc_g+\Hess_gf_g}_{L^2(M,e^{-f_g}dV_g)}\leq \norm{\Rc_g}_{L^2(M,e^{-f_g}dV_g)}
\end{equation}
follows immediatly from the $L^2(M,e^{-f_g}dV_g)$-orthogonality of $\Rc+\Hess f$ and $\Hess f$ (see (\ref{dlambdadiffinv})).
\end{remark}

\begin{proof}[Proof of Theorem B]
We can assume  $g\in \grf+\ker\div_{\grf}, g=\bar{g}+h, \bar{g}\in\mathcal{P}_{ \grf}, h\in N_{ \grf}$, arguing as in the proof of Theorem C. Then, always working in a small enough $C^{2,\alpha}$-neighborhood,
\begin{equation}
|\lambda(g)|\leq C \norm{h}_{H^2}^2.
\end{equation}
This estimate follows from $\lambda(\bar{g})=0$, $D\lambda(\bar{g})=0$ and (\ref{pertseries2}). Together with (\ref{toshow}), the theorem follows.
\end{proof}

\begin{remark} To show convergence of a parabolic gradient flow, $\tfrac{d}{dt}g=\nabla\lambda(g)$, starting near a local maximizer  $g_{\mathrm{max}}$ of its energy $\lambda$, an inequality of the form $\norm{\nabla\lambda(g)}_{L^2}\geq c |\lambda(g)-\lambda(g_{\mathrm{max}})|^{1-\theta}$ for some $\theta\in(0,\tfrac{1}{2}]$ is sufficient \cite{Si}. Let us also remark that from Perelman's evolution inequality $\tfrac{d\lambda}{dt}\geq \tfrac{2}{n}\lambda^2$ we only get the inequality for $\theta=0$.
\end{remark}

\section{Stability and Instability under Ricci flow}\label{stabandinstab}

Let $(M^n,\grf)$ be compact Ricci-flat. Assume all infinitesimal Ricci-flat deformations of $\grf$ are integrable and $\Lap^L_{\grf}\leq 0$ on TT. Let $k\geq 3$.\\
By the Theorems A, B, C and Lemma \ref{ctwoalphalemma}, there exist constants $\eps_0>0$ and $C_1,C_2<\infty$ such that for all $g$ with $\norm{g-\grf}_{C^k_{\grf}}<\eps_0$:
\begin{align}
\lambda(g)\leq 0\quad \text{and}\quad \lambda(g)=0 \Leftrightarrow \Rc_g=0\label{stabA}\\
|\lambda(g)|^{1/2}\leq C_1 \norm{\Rc_g+\Hess_gf_g}_{L^2_{f_g}}\label{stabB}\\
\norm{\Rc_g}_{L^2_g}\leq C_2 \norm{\Rc_g+\Hess_gf_g}_{L^2_{f_g}}\label{stabC}
\end{align}
Here, we define the $C^k$-norm using $\grf$ and the $L^2_f$-norm using the metric $g$ and the measure $e^{-f_g}dV_g$.

\begin{lemma}[Energy controls the distance]\label{energylemma}
Let $(M^n,\grf)$ and $k,\eps_0,C_1,C_2$ as above. Let $0\leq t_1< t_2<T$ and $g(t)$ a Ricci flow (\ref{ricciflow}) with $\norm{g(t)-\grf}_{C^k_{\grf}}<\eps_0$ for all $t\in[0,T)$. Then
\begin{equation}
\int_{t_1}^{t_2}\norm{\Rc_{g(t)}}_{L^2_{g(t)}}dt \leq C_1C_2 \left(|\lambda(g(t_1))|^{\tfrac{1}{2}}-|\lambda(g(t_2))|^{\tfrac{1}{2}}\right).
\end{equation}
\end{lemma}

\begin{proof}
Without loss of generality, we can assume the inequality in (\ref{stabA}) is strict, i.e. $\lambda(g(t))<0$ for all $t\in[0,T)$. By Perelman's monotonicity formula $\lambda=-|\lambda|$ is increasing along the flow, more precisely,
\begin{align}
-\tfrac{d}{dt}|\lambda(g(t))|^{1/2}&=\tfrac{1}{2}|\lambda(g(t))|^{-1/2}\tfrac{d}{dt}\lambda(g(t))\nonumber\\
&=|\lambda(g(t))|^{-1/2}\langle \Rc_{g(t)}+\Hess_{g(t)}f_{g(t)},\Rc_{g(t)}\rangle_{L^2_{f_{g(t)}}}\nonumber\\
&=|\lambda(g(t))|^{-1/2}\norm{\Rc_{g(t)}+\Hess_{g(t)}f_{g(t)}}_{L^2_{f_{g(t)}}}^2\nonumber\\
&\geq \tfrac{1}{C_1C_2}\norm{\Rc_{g(t)}}_{L^2_{g(t)}}
\end{align}
where we used (\ref{1stvaroflambda}), (\ref{dlambdadiffinv}), (\ref{stabB}) and (\ref{stabC}). This proves the lemma.
\end{proof}

\begin{lemma}[Estimates for $t\leq 1$]\label{tleq1est}
Let $(M^n,\grf)$ be compact, Ricci-flat, $k\geq 3, \eps>0$. Then there exists a $\delta_1=\delta_1(M^n,\grf,\eps,k)>0$ such that:
If $\norm{g_0-\grf}_{C^{k+2}_{\grf}}<\delta_1$ then the Ricci flow starting at $g_0$ exists on $[0,1]$ and satisfies
\begin{equation}
\norm{g(t)-\grf}_{C^k_{\grf}}<\eps \quad \forall t\in[0,1].
\end{equation}
\end{lemma}

\begin{proof}
From $\partial_t\Rm=\Lap\Rm+\Rm\ast\Rm$ and $\partial_t\Rc=\Lap\Rc+\Rm\ast\Rc$, we get the evolution inequalities
\begin{align}
\partial_t |D^i\Rm|^2&\leq\Lap |D^i\Rm|^2+\sum_{j=0}^iC_{ij}|D^{i-j}\Rm||D^j\Rm||D^i\Rm|,\label{evofdrm}\\
\partial_t |D^i\Rc|^2&\leq\Lap |D^i\Rc|^2+\sum_{j=0}^iC_{ij}|D^{i-j}\Rm||D^j\Rc||D^i\Rc|.\label{evofdrc}
\end{align}
From (\ref{evofdrm}), by the maximum principle, there exists a $\tilde{K}=\tilde{K}(K,n,k)<\infty$ such that
if $g(t)$ is a Ricci flow on [0,T] with $T\leq 1$ and
\begin{equation}
|\Rm(x,t)|\leq K,\quad |D^i\Rm(x,0)|\leq K, \quad \forall x\in M, t\in[0,T],i\leq k
\end{equation}
then
\begin{equation}
|D^i\Rm(x,t)|\leq \tilde{K} \quad \forall x\in M, t\in[0,T],i\leq k.
\end{equation}
From (\ref{evofdrc}), by the maximum principle, for every $\tilde{\eps}>0$, there exists a $\tilde{\delta}=\tilde{\delta}(\tilde{K},\tilde{\eps},n,k)>0$ such that for $g(t)$ as above:
\begin{align}
&|D^i\Rc(x,0)|\leq \tilde{\delta} \quad \forall x\in M, i\leq k\nonumber\\
&\quad\Rightarrow\quad |D^i\Rc(x,t)|\leq \tilde{\eps} \quad  \forall x\in M, t\in[0,T],i\leq k.
\end{align}
Finally, as long as the $C^k$-norms defined via $\grf$ and $g(t)$ differ at most by a factor $2$,
\begin{equation}
\tfrac{d}{dt}\norm{g(t)-\grf}_{C^{k}_{\grf}}
\leq\norm{2\Rc_{g(t)}}_{C^{k}_{\grf}}
\leq 4 \sum_{i=0}^k\sup_{x\in M}  |D^i\Rc(x,t)|.
\end{equation}
Now, we put the above facts together:
Without loss of generality, assume $\eps>0$ is small enough that the $C^k$-norms defined via $\grf$ and via $g$ with $\norm{g-\grf}_{C^{k}_{\grf}}\leq\eps$ differ at most by a factor $2$. Pick some small enough $\bar{\delta}>0$. Define
\begin{multline}
K:=\sup \{|\Rm_g(x)|;\; \norm{g-\grf}_{C^{k}_{\grf}}\leq\eps, x\in M\}\\
+\sup \{|D^i\Rm_g(x)|;\; \norm{g-\grf}_{C^{k+2}_{\grf}}\leq\bar{\delta},x\in M,\,i\leq k\}
<\infty.
\end{multline}
Let $\tilde{K}:=\tilde{K}(K,n,k), \tilde{\delta}:=\tilde{\delta}(\tilde{K},\tfrac{\eps}{16(k+1)},n,k)$ and let $\delta_1<\bar{\delta}$ be so small that
\begin{equation}
\norm{g-\grf}_{C^{k+2}_{\grf}}\leq\delta_1\Rightarrow \sup_{x\in M,\,i\leq k}|D^i\Rc_g(x)|\leq\tilde{\delta},\quad \norm{g-\grf}_{C^{k}_{\grf}}\leq\tfrac{\eps}{4}.
\end{equation}
Let $\norm{g_0-\grf}_{C^{k+2}_{\grf}}<\delta_1$. Let $T\in(0,\infty]$ be the maximal time such that the Ricci flow starting at $g_0$ exists on [0,T) and satisfies
\begin{equation}
\norm{g(t)-\grf}_{C^{k}_{\grf}}<\eps\quad\forall t\in[0,T).
\end{equation}
Suppose, towards a contradiction, $T\leq1$. Then
\begin{equation}
\norm{g(t)-\grf}_{C^{k}_{\grf}}\leq \norm{g_0-\grf}_{C^{k}_{\grf}}+4(k+1)\sup_{\begin{subarray}{l}
x\in M,\,i\leq k\\
t\in[0,T] \end{subarray}}
|D^i\Rc(x,t)|\leq\tfrac{\eps}{2}.
\end{equation}
for all $t\in[0,T]$. This contradicts the maximality in the definition of $T$ and  proves the lemma.
\end{proof}

\begin{lemma}[Estimates for $t\geq 1$] \label{tgeq1est}
Let $(M^n,\bar{g})$ be compact and $\eps>0$ small enough. Then there exist constants $C_i=C_i(M^n,\bar{g},\eps,i)<\infty$ such that if 
g(t) is a Ricci flow with $\norm{g(t)-\bar{g}}_{C^2_{\bar{g}}}<\eps$ for all $t\in[0,T)$ then
\begin{equation}
\norm{\Rc_{g(t)}}_{C^i_{g(t)}}\leq C_i \norm{\Rc_{g(t-1/2)}}_{L^2_{g(t-1/2)}} \quad \forall t\in[1,T).
\end{equation}
\end{lemma}

\begin{proof}
Since $\eps$ is small enough, we have uniform curvature bounds and a uniform bound for the Sobolev constant. Thus, from the evolution inequality
\begin{equation}
\partial_t|\Rc|^2\leq\Lap|\Rc|^2+CK|\Rc|^2,
\end{equation}
by Moser iteration (see e.g. \cite{Ye}), there exists $\tilde{K}=\tilde{K}(M,\bar{g},\eps)<\infty$ such that
\begin{equation}
\sup_{x\in M}|\Rc(x,t)|\leq\tilde{K}\norm{\Rc_{g(t-1/4)}}_{L^2_{g(t-1/4)}}.
\end{equation}
Note that usually a spacetime integral appears on the right hand side, however one can get rid of the time integral using
\begin{equation}
\tfrac{d}{dt}\int_M{|\Rc|^2}dV\leq \tilde{C}K\int_M{|\Rc|^2}dV.
\end{equation}
From the evolution equation of $\Rm$, by Bando-Bernstein-Shi estimates (see e.g. \cite{Ha2}), one gets uniform bounds for the derivatives of $\Rm$ let's say on $[\tfrac{3}{4},T)$. Using these bounds in the evolution equation of $\Rc$, again by Bando-Bernstein-Shi type estimates, we get constants $\bar{K}_i=\bar{K}_i(M,\bar{g},\eps,i)<\infty$ such that
\begin{equation}
\sup_{x\in M,j\leq i}|D^j\Rc(x,t)|\leq\bar{K}_i\sup_{x\in M}|\Rc(x,t-\tfrac{1}{4})|\qquad \forall t\in[1,T).
\end{equation}
This proves the lemma.
\end{proof}

Let us restate Theorem E in the following equivalent form.

\begin{thmE}[Dynamical stability]
Let $(M^n,\grf)$ be compact, Ricci-flat and $k\geq 3$. Assume that all infinitesimal Ricci-flat deformations of $\grf$ are integrable and $\Lap^L_{\grf}\leq 0$ on TT.\\
Then for every $\eps>0$ there exists a $\delta=\delta(M^n,\grf,\eps,k)>0$ such that if $\norm{g_0-\grf}_{C^{k+2}_{\grf}}<\delta$, then the Ricci flow starting at $g_0$ exists on $[0,\infty)$, satisfies $\norm{g(t)-\grf}_{C^k_{\grf}}<\eps$ for all $t\in[0,\infty)$ and $g(t)\to g_{\infty}$ exponentially as $t\to\infty$, with $\Rc_{g_{\infty}}=0$ and $\norm{g_{\infty}-\grf}_{C^k_{\grf}}<\eps$.
\end{thmE}

\begin{proof}[Proof of Theorem E]
Without loss of generality, assume $\eps$ is small enough that the previous lemmas apply and that the $C^k$-norms defined via $g$ and $\grf$ with $\norm{g-\grf}_{C^k_{\grf}}<\eps$ differ at most by a factor $2$. Let $\delta:=\min\{\delta_1,\delta_2\}$, where $\delta_1=\delta_1(M,\grf,\tfrac{\eps}{4},k)>0$ is from Lemma \ref{tleq1est} and $\delta_2=\delta_2(M,\grf,\eps,k)>0$ is such that $\norm{g_0-\grf}_{C^{k+2}_{\grf}}<\delta_2$ implies
\begin{equation}
4C_1C_2C_k|\lambda(g_0)|^{1/2}\leq\tfrac{\eps}{4}
\end{equation}
where $C_k=C_k(M^n,\grf,\eps,k)$ is from Lemma \ref{tgeq1est} and $C_1,C_2$ are from the beginning of Section \ref{stabandinstab}.
Let $\norm{g_0-\grf}_{C^{k+2}_{\grf}}<\delta$ and $T\in(1,\infty]$ be the maximal time such that the Ricci flow starting at $g_0$ satisfies
\begin{equation}
\norm{g(t)-\grf}_{C^k_{\grf}}<\eps \quad \forall t\in[0,T).
\end{equation}
Without loss of generality, assume the inequality $\lambda(g(t))\leq0$ is strict for all $t\in[0,T)$. Suppose towards a contradiction $T<\infty$. Then for all $t\in[1,T)$
\begin{equation}
\tfrac{d}{dt}\norm{g(t)-g(1)}_{C^k_{\grf}}\leq 4 \norm{\Rc_{g(t)}}_{{C^k_{g(t)}}}\leq 4C_k \norm{\Rc_{g(t-1/2)}}_{L^2_{g(t-1/2)}}
\end{equation}
by Lemma \ref{tgeq1est}, and thus by Lemma \ref{tleq1est} and Lemma \ref{energylemma}
\begin{equation}
\norm{g(t)-\grf}_{C^k_{\grf}}\leq \norm{g(1)-\grf}_{C^k_{\grf}}+4C_1C_2C_k|\lambda(g_0)|^{1/2}\leq\tfrac{\eps}{2}
\end{equation}
for all $t\in[1,T)$. This contradicts the maximality in the definition of $T$, thus $T=\infty$ and
\begin{align}
&\norm{g(t)-\grf}_{C^k_{\grf}}<\eps, \quad \quad t\in[0,\infty)\\
&\int_0^{\infty}\norm{\dot{g}(t)}_{C^k_{\grf}}dt <\infty.
\end{align}
Thus $g(t)\to g_{\infty}$ in $C^k_{\grf}$ for $t\to\infty$ (since $g(t)$ is a Ricci flow with smooth initial metric, the convergence is in fact smooth).
Along the flow, we have 
\begin{align}
&-\tfrac{d}{dt}|\lambda|=2\norm{\Rc+\Hess f}_{L^2_f}^2\geq\tfrac{2}{C_1^2}|\lambda|\\
&\qquad\Rightarrow |\lambda(g(t_2))|\leq e^{-2(t_2-t_1)/C_1^2}|\lambda(g(t_1))|.
\end{align}
Thus $\lambda(g_\infty)=0$, $\Rc_{g_\infty}=0$ and, using in particular Lemma \ref{energylemma} and Lemma \ref{tgeq1est}, we see that the convergence is exponential (the exponential convergence is a consequence of the optimal {\L}ojasiewicz exponent $\tfrac{1}{2}$). This proves the theorem.
\end{proof}

\begin{remark}
Since the Ricci flow is not strictly parabolic, we mostly worked with the evolution equations of the curvatures. This is the reason for the loss of two derivatives in Theorem E. For the Ricci-DeTurck flow one of course gets optimal regularity. However, when translating back to the Ricci flow, one also loses two derivatives.
\end{remark}

\begin{proof}[Proof of Theorem F]
Pick a sequence of metrics $g_i^0\to \grf$ in $C^\infty$ with $\lambda(g_i^0)>0$.
Let $\tilde{g}_i(t)$ be the Ricci flows starting at $g_i^0$. Since $\lambda(g_i^0)>0$, by Perelman's evolution inequality $\tfrac{d\lambda}{dt}\geq\tfrac{2}{n}\lambda^2$, the flows become singular in finite time. Since $g_i^0\to \grf$ in $C^\infty$, the flows exist and stay inside a small ball for longer and longer times. Let $\eps>0$ be small enough. Let $t_i$ be the first time when $d_{C^\infty}(\tilde{g}_i(t),\grf)=\eps$. Then $t_i\to\infty$ and, always assuming $i$ is large enough,
\begin{equation}
\tfrac{\eps}{2}\leq
d_{C^\infty}(\tilde{g}_i(t_i),\tilde{g}_i(1))
\leq C\lambda(\tilde{g}_i(t_i))^{1/2},
\end{equation}
by the {\L}ojasiewicz inequality, the gauge estimate and parabolic estimates. Thus $\lambda(\tilde{g}_i(t_i))\geq c>0$, which will be used to exclude trivial solutions.\\
Shifting time, we obtain a family of Ricci flows $g_i(t):=\tilde{g}_i(t+t_i), t\in[-t_i,T),-t_i\to-\infty, T>0$ with
\begin{align}
&d_{C^\infty}(g_i(t),\grf)\leq 2\eps\quad\forall t\in[-t_i,T),\label{cinftybounds}\\
&\lambda(g_i(0))\geq c>0,\\
&g_i(-t_i)=g_i^0\to \grf \quad\text{in}\; C^\infty.
\end{align}
From (\ref{cinftybounds}) and the Ricci flow equation, we have $C^\infty$ space-time bounds. Thus, after passing to a subsequence,  $g_i$ converges to an ancient Ricci flow $g$ in $C^\infty_{\mathrm{loc}}(M\times (-\infty,T))$ with $\lambda(g(0))\geq c>0$. In particular, this implies that $g$ is nontrivial and becomes singular in finite time. Moreover, $\lambda(g(t))\geq 0$ for all $t\in(-\infty,T)$. Finally, for $-t_i\leq t$,
\begin{align}
d_{C^\infty}(\grf,g(t))&\leq d_{C^\infty}(\grf,g_i^0)+d_{C^\infty}(g_i(-t_i),g_i(t))+d_{C^\infty}(g_i(t),g(t))\nonumber\\
&\leq d_{C^\infty}(\grf,g_i^0)+C\lambda(g_i(t))^{1/2}+d_{C^\infty}(g_i(t),g(t))
\end{align}
by the {\L}ojasiewicz inequality, the gauge estimate and parabolic estimates. Since $\lambda(g_i(t))$ is bounded up to time zero and $\tfrac{d\lambda}{dt}\geq\tfrac{2}{n}\lambda^2$, we see that $\lambda(g_i(t))$ is very small for very negative $t$. Thus $g(t) \to \grf$ in $C^\infty$ as $t\to-\infty$ and this finishes the proof of the theorem.
\end{proof}

\begin{remark}\label{technicalremark}
In general, {\L}ojasiewicz type inequalities find their truest applications in the nonintegrable case. In fact, the conclusions of  Theorem E and F hold under the slightly weaker assumption that $\lambda$ is maximal respectively nonmaximal and $\grf$ satisfies the {\L}ojasiewicz type inequalities
\begin{align}
&\norm{\Rc_g+\Hess_gf_g}_{L^2}\geq c |\lambda(g)|^{1-\theta_1}\label{generalB}\\
&\norm{\Rc_g+\Hess_gf_g}_{L^2}^{\theta_2}\geq c \norm{\Rc_g}_{L^2}\label{generalC}
\end{align}
for $\theta_1\in (0,\tfrac{1}{2}], \theta_2\in (0,1]$ with
\begin{equation}
2\theta_1+\theta_2-\theta_1\theta_2>1.\label{thetacondition}
\end{equation}
The estimates (\ref{generalB}) and (\ref{generalC}) should be the natural generalizations of Theorem B and C to the nonintegrable case. It is an interesting problem to prove them using Lyapunov-Schmidt reduction, the finite-dimensional {\L}ojasiewicz inequalities and the estimates for the error terms developed in this article. Note however, that the condition (\ref{thetacondition}) is essentially uncheckable, so a new way of dealing with the gauge problem should be found.
\end{remark}

\appendix

\section{Proof of Lemma 2.1}\label{proofpertlemma}

We expand
\begin{equation}
H_{g(\eps)}=H+\eps H'[h]+\tfrac{\eps^2}{2}H''[h,h]+\tfrac{\eps^3}{6}H'''[h,h,h]+O(\eps^4),
\end{equation}
and
\begin{align}
&(\lambda-H_{g(\eps)})^{-1}=(\lambda-H)^{-1}+\eps (\lambda-H)^{-1} H'[h] (\lambda-H)^{-1}\\
&\quad+\tfrac{\eps^2}{2}\left((\lambda-H)^{-1}H''[h,h]+2\left((\lambda-H)^{-1}H'[h]\right)^2\right)(\lambda-H)^{-1}+O(\eps^3).\nonumber
\end{align}
We insert this in (\ref{pertform1}), use $(\lambda-H)^{-1}w=(\lambda-\lambda(g))^{-1}w$, use that $H$ is symmetric with respect to the $L^2$-inner product and that $w$ is normalized. Thus
\begin{equation}
\frac{1}{\langle w, P_{g(\eps)} w\rangle} = 1 - \eps^2\tfrac{1}{2\pi i}\oint \langle w, H'[h](\lambda-H)^{-1}H'[h]w\rangle\tfrac{d\lambda}{(\lambda-\lambda(g))^2}+O(\eps^3),
\end{equation}
\begin{align}
&\langle w,(H_{g(\eps)}-H)P_{g(\eps)} w\rangle\nonumber\\
&\begin{array}{l}
\quad=\eps\langle w,H'[h]w\rangle\\
\qquad+\tfrac{\eps^2}{2}\left(\langle w,H''[h,h] w\rangle+\tfrac{2}{2\pi i}\oint \langle w, H'[h](\lambda-H)^{-1}H'[h]w\rangle\tfrac{d\lambda}{\lambda-\lambda(g)}\right)\\
\qquad+\tfrac{\eps^3}{6}\langle w,H'''[h,h,h] w\rangle\\
\qquad+\tfrac{\eps^3}{6}\tfrac{6}{2\pi i}\oint \langle w, H'[h](\lambda-H)^{-1}H'[h](\lambda-H)^{-1}H'[h]w\rangle\tfrac{d\lambda}{\lambda-\lambda(g)}\\
\qquad+\tfrac{\eps^3}{6}\tfrac{3}{2\pi i}\oint \langle w, H'[h](\lambda-H)^{-1}H''[h,h]w\rangle\tfrac{d\lambda}{\lambda-\lambda(g)}\\
\qquad+\tfrac{\eps^3}{6}\tfrac{3}{2\pi i}\oint \langle w, H''[h](\lambda-H)^{-1}H'[h]w\rangle\tfrac{d\lambda}{\lambda-\lambda(g)}+O(\eps^4),
\end{array}
\end{align}
and the formulas in the lemma follow from (\ref{pertform2}).\\
Let us now justify convergence and analyticity. We have a family of closed operators
\begin{equation}
H(\eps)=-4\Lap_{g+\eps h}+R_{g+\eps h}:H^2(M)\subset L^2(M,dV_g)\to  L^2(M,dV_g).
\end{equation}
For every $u\in L^2(M,dV_g)$ and $v\in H^2(M)$, the $L^2(M,dV_g)$-inner product $\langle u,H(\eps)v\rangle$ depends analytically on $\eps$. Since every weakly analytic function is strongly analytic, for every $v\in H^2(M)$, $\eps\mapsto H(\eps)v$ is an $L^2(M,dV_g)$-valued analytic function. By the above, $H(\eps)$ is an analytic family of type (A) and thus an analytic family in the sense of Kato \cite[Sec. XII.2]{RS}. Therefore, the smallest eigenvalue $\lambda(g(\eps))$ is an analytic function of $\eps$ by the Kato-Rellich theorem \cite[Thm. XII.8]{RS}. Finally, by \cite[Thm. XII.7]{RS} the function
\begin{equation}
(\lambda,\eps)\mapsto (\lambda-H(\eps))^{-1}
\end{equation}
is an $\mathcal{L}(L^2(M,dV_g))$-valued analytic function of two variables defined on an open set, say on $\{(\lambda,\eps)\in\mathbb{C}^2:r-\delta<|\lambda-\lambda(g)|<r+\delta, |\eps|<\delta\}$, and this justifies the above computations.

\section{The second variation}\label{appendixsecondvariation}
Let $(M,\grf)$ be compact, Ricci-flat,  $h\in\ker\div_{\grf}$. 
We would like to evaluate (\ref{pertseries2}). Since $w$ is constant, there are some simplifications. To get $H''[h,h]$ we compute
\begin{equation}
\tfrac{d^2}{d\eps^2}|_0R_{g(\eps)}=\tfrac{d}{d\eps}|_0g^{ij}g^{kl}(-h_{ik}R_{jl}+D_iD_k h_{jl} -D_iD_j h_{kl}).
\end{equation}
There are contributions from the derivative of $g^{-1}$ (first line), $\Rc$ (second line) and $D$ (third line) respectively. Using $\div h=0, \Rc=0$, which implies in particular $-D_iD_kh_{il}=R_{kplq}h_{pq}$, we obtain
\begin{align}
&\tfrac{d^2}{d\eps^2}|_0R_{g(\eps)}\nonumber\\
&\quad=h_{ij}\Lap h_{ij}+R_{ipjq}h_{ij}h_{pq}+h_{ij}D_iD_j\tr h\nonumber\\
&\qquad+\tfrac{1}{2}h_{ij}\Lap h_{ij}+R_{ipjq}h_{ij}h_{pq}+\tfrac{1}{2}h_{ij}D_iD_j\tr h\nonumber\\
&\qquad+|D h|^2-\tfrac{1}{2}|D \tr h|^2-D_ih_{jk}D_kh_{ij}+\tfrac{1}{2}D_i(h_{jk}D_i h_{jk}+h_{ij}D_j \tr h)\nonumber\\
&\quad=2\langle h,\Lap h\rangle+\tfrac{3}{2}|D h|^2-\tfrac{1}{2}|D \tr h|^2+2\langle h, D^2\tr h\rangle\nonumber\\
&\qquad+2R_{ipjq}h_{ij}h_{pq}-D_ih_{jk}D_kh_{ij}.
\end{align}
Together with $\tfrac{d^2}{d\eps^2}|_0\Lap_{g(\eps)}1=0$, after partial integration, commuting the derivatives in $D_iD_k h_{ij}$ and using $\div h=0$ again, we get
\begin{equation}
\langle 1, H''[h,h]1\rangle=\tfrac{1}{2}\int_M \left(\langle h, \Lap^L h\rangle + \tr h\Lap \tr h\right) dV.
\end{equation}
The other term contributing to the second variation is proportional to
\begin{equation}
\tfrac{2}{2\pi i}\oint_{|\lambda|=r}\langle 1,H'[h](\lambda-H)^{-1}H'[h]1\rangle\tfrac{d\lambda}{\lambda}.
\end{equation}
Now, we insert $H'[h]$ from (\ref{hprime}). Since $\Rc=0, \div h=0$ and $D1$=0, many terms vanish. After a partial integration, even more terms vanish and we obtain
\begin{align}
&\tfrac{2}{2\pi i}\oint_{|\lambda|=r}\langle 1,H'[h](\lambda-H)^{-1}H'[h]1\rangle\tfrac{d\lambda}{\lambda}\nonumber\\
&\quad=-\tfrac{2}{2\pi i}\oint_{|\lambda|=r} \langle\tr h, \Lap (\lambda+4\Lap)^{-1} \Lap \tr h\rangle \tfrac{d\lambda}{\lambda} =-\tfrac{1}{2}\int_M \tr h\Lap \tr h\, dV.
\end{align}
To justify the last step, note that $\Lap (\lambda+4\Lap)^{-1}$ converges to $\tfrac{1}{4}$ as $\lambda$ tends to zero. Finally, $w=\mathrm{Vol}(M)^{-1/2}$, $\grf$ is a critical point of $\lambda$, and $\lambda$ is invariant under diffeomorphism. This proves (\ref{2ndvaroflambda}).

\end{document}